\documentclass[11pt,oneside]{amsart}
\usepackage{amsmath,amsthm,mathtools}
\usepackage[margin=1.25in]{geometry}

\usepackage{xcolor}
\usepackage[
colorlinks = true,
linkcolor = blue!60!black,  
citecolor = green!50!black, 
urlcolor = blue!60!black,   
bookmarksnumbered = true,
pdfusetitle
]{hyperref}

\theoremstyle{plain}
\newtheorem{theorem}{Theorem}[section]

\newtheorem{corollary}[theorem]{Corollary}
\newtheorem{lemma}[theorem]{Lemma}
\newtheorem{conjecture}[theorem]{Conjecture}

\theoremstyle{definition}
\newtheorem{definition}[theorem]{Definition}
\newtheorem{example}[theorem]{Example}

\theoremstyle{remark}
\newtheorem*{note}{Note}
\newtheorem*{remark}{Remark}

\newtheorem{openproblem}[theorem]{Open Problem}

\newcommand{\R}{\mathbb{R}}

\newcommand{\C}{\mathbb{C}}
\DeclareMathOperator{\im}{Im}

\DeclareMathOperator{\LP}{\mathcal{LP}}

\begin{document}

\title{Complex zero strip decreasing operators}
\author{David A. Cardon}
\address{Department of Mathematics, Brigham Young University, Provo, UT
84602}
\email{cardon@mathematics.byu.edu}
\date{October 20, 2026}
\keywords{zeros of entire functions, Laguerre-P\'olya class, complex zero
strip decreasing operators} \subjclass[2010]{30C15, 47B38}

\begin{abstract}
  Let $\phi(z)$ be a function in the Laguerre-P\'olya class. Write
  $\phi(z)=e^{-\alpha z^2} \phi_1(z)$ where $\alpha \geq 0$ and where
  $\phi_1(z)$ is a real entire function of genus $0$ or $1$. Let $f(z)$ be
  any real entire function of the form $f(z)=e^{-\gamma z^2}f_1(z)$ where
  $\gamma \geq 0$ and $f_1(z)$ is a real entire function of genus $0$ or
  $1$ having all of its zeros in the strip
  $S(r)=\{z \in \C \colon -r \leq \im z \leq r\}$, where $r>0$. If
  $\alpha\gamma<1/4$, the linear differential operator $\phi(D)f(z)$,
  where $D$ denotes differentiation, is known to converge to a real entire
  function whose zeros also belong to the strip $S(r)$. We describe
  several necessary and sufficient conditions on $\phi(z)$ such that all
  zeros of $\phi(D)f(z)$ belong to a smaller strip
  $S(r_1)=\{z \in \C \colon -r_1 \leq \im z \leq r_1\}$ where
  $0 \leq r_1 < r$ and $r_1$ depends on $\phi(z)$ but is independent of
  $f(z)$. We call a linear operator having this property a \textit{complex
  zero strip decreasing operator} or CZSDO. We examine several relevant
  examples, in certain cases we give explicit upper and lower bounds for
  $r_1$, and we state several conjectures and open problems regarding
  complex zero strip decreasing operators.
\end{abstract}

\maketitle

\begin{note}
  This version corrects the preprint
  \href{https://arxiv.org/abs/1401.0237v2}{arXiv:1401.0237v2} and the paper
  in J. Math. Anal. Appl. 426 (2015), 406--422. Typographical errors have been corrected, several errors and
  gaps in the proofs have been fixed, and the statement of
  Conjecture~\ref{conjecture:twosidedbounds} has been reworded to express
  its intended meaning. A paragraph after
  Definition~\ref{definition:LPtothegammaofr} explains the broader use of
  the term CZSDO in Sections~1, \ref{section:examples}
  and~\ref{section:futherstudy}. Conjectures~\ref{conjecture:lowerdensity}
  and~\ref{conjecture:MainConjecture} of the published version are false,
  and a counterexample is given in the remark following
  Conjecture~\ref{conjecture:MainConjecture}. The numbering of the
  preprint has been preserved.
\end{note}


\section{Introduction}
An important problem in the theory of the distribution of zeros of a
collection of entire functions is to understand the effect of linear
operators that act on the collection. It is particularly interesting when
the operators preserve a nice property about the location of the zeros.
The linear operators we will study in this paper are differential
operators $\phi(D)$ where $\phi(z)$ is a function in the Laguerre-P\'olya
class and $D$ is differentiation. If $f(z)$ is a real entire function
satisfying appropriate technical requirements whose zeros belong to the
strip $S(r)=\{ z
\in \C \colon -r \leq \im z \leq r\}$, we study the problem of when all zeros of
$\phi(D)f(z)$ belong to a smaller strip $S(r_1)$ where $0 \leq r_1 < r$.
The main results in the paper are stated in
Theorems~\ref{theorem:MainTheorem:1} and~\ref{theorem:MainTheorem:2}.

Before stating these theorems we will need a few definitions and a
technical lemma that defines the linear differential operator $\phi(D)$
and tells us when the expression $\phi(D)f(z)$ makes sense.

\begin{definition}[$\LP$ and $\LP_1$]
  The \textit{Laguerre-P\'olya class}, denoted $\LP$, consists of the real
  entire functions whose Weierstrass product representations are of the
  form
  \begin{equation} \label{eqn:LaguerrePolyaWeierstrassProduct}
    c z^m e^{\alpha z - \beta z^2} \prod_k \left(1 -\frac{z}{\alpha_k} \right) e^{z/\alpha_k},
  \end{equation}
  where $c$, $\alpha$, $\beta$, $\alpha_k$ are real, $\alpha_k \neq 0$,
  $\beta \geq 0$, $m$ is a nonnegative integer, and
  $\sum_k |\alpha_k|^{-2} < \infty$. The subclass $\LP_1$ of $\LP$
  consists of those functions in $\LP$ with $\beta=0$ in
  Eq.~\eqref{eqn:LaguerrePolyaWeierstrassProduct}.
\end{definition}

The class $\LP$ consists of the entire functions obtained as uniform
limits on compact sets of sequences of real polynomials having only real
zeros. See Levin~\cite[Thm.~3,~p.~331]{Levin1980}. Motivation for why this
class of functions naturally arises in relation to differential operators
is given in Section~\ref{section:philosophy}.

\begin{definition}[$\LP(r)$ and $\LP_1(r)$]
  For $r \geq 0$, the \textit{extended Laguerre-P\'olya class}, denoted
  $\LP(r)$, consists of the real entire functions having the Weierstrass
  product representation in
  Eq.~\eqref{eqn:LaguerrePolyaWeierstrassProduct} except that the zeros
  belong to the strip
  \[
    S(r) = \{z \in \C \colon -r \leq \im z \leq r\}.
  \]
  Thus, the zeros of a nonzero function $f(z) \in \LP(r)$ are either real
  or occur in complex conjugate pairs. The subclass $\LP_1(r)$ of $\LP(r)$
  consists of those functions in $\LP(r)$ with $\beta=0$ in
  Eq.~\eqref{eqn:LaguerrePolyaWeierstrassProduct}. If $r<0$ or $r$ is
  imaginary, we define $\LP(r)=\LP$ and $S(r)=\R$.
\end{definition}

We include the zero function in each of these classes. Statements about
the zeros of the image under an operator allow the image to be identically
zero. The functions $\phi$ defining differential operators will be assumed
to be nonzero.

The following lemma shows how functions in $\LP$ define linear
differential operators on functions in $\LP(r)$. A trivial modification to
the proof of a theorem in Levin~\cite{Levin1980} gives:

\begin{lemma}[Levin~\cite{Levin1980}, Thm.~8, p.~360]  \label{lemma:differentialoperator}
  Assume
  \[
    \phi(z) = e^{-\gamma_1 z^2}\phi_1(z) = \sum_{k = 0}^{\infty} a_k z^k \in
    \LP
  \]
  where $\gamma_1 \geq 0$ and $\phi_1(z)\in\LP_1$. Also let $r\geq 0$ and
  assume $f(z)=e^{-\gamma_2 z^2} f_1(z) \in \LP(r)$ where
  $\gamma_2 \geq 0$ and $f_1(z)\in\LP_1(r)$. If $\gamma_1 \gamma_2 <1/4$,
  the linear differential operator $\phi(D)$ is defined by
  \begin{equation} \label{eqn:differentialoperator}
    \phi(D)f(z) = \sum_{k = 0}^{\infty} a_k f^{(k)}(z),
  \end{equation}
  where $D$ denotes differentiation. The sum converges uniformly on every
  compact subset of $\C$ and $\phi(D)f(z) \in \LP(r)$.
\end{lemma}

The convergence estimates in Levin's proof apply here as well: for every
$\epsilon>0$, the canonical products give bounds
$|\phi(z)|\leq C_\epsilon e^{(\gamma_1+\epsilon)|z|^2}$ and
$|f(z)|\leq C_\epsilon e^{(\gamma_2+\epsilon)|z|^2}$. Choose $\epsilon$
sufficiently small that $(\gamma_1+\epsilon)(\gamma_2+\epsilon)<1/4$.
Cauchy's inequalities then give an absolutely and uniformly convergent
majorant for \eqref{eqn:differentialoperator} on each compact set. To see
the assertion about the strip, first approximate $f$ by real polynomials
with zeros in $S(r)$. Such approximations are obtained from finite canonical products, in which
each nonreal zero is taken together with its complex conjugate, by
approximating their real exponential and Gaussian factors by polynomials
with real zeros. The function $f(z-ir)$ is
consequently a uniform limit on compact sets of polynomials having no
zeros in the open lower half-plane, as is $\phi(z)$. Levin's theorem is
stated for the class $P^{*}$ of entire functions that are such limits
\cite[Thm.~4,~p.~334]{Levin1980}. Every function in $P^{*}$ has the form
$e^{-\gamma z^{2}}F(z)$, where $\gamma \geq 0$ and $F(z)$ is an entire
function of genus at most $1$, and has no zeros in the open lower
half-plane. Translation preserves this form and the constant $\gamma$.
Thus $\phi(z)$ and $f(z-ir)$ belong to $P^{*}$ with the constants
$\gamma_{1}$ and $\gamma_{2}$ of the lemma, and Levin's theorem shows that
$\phi(D)f(z-ir)$ belongs to $P^{*}$. Since $\phi(D)$ commutes with
translation, the real entire function $\phi(D)f(z)$ has the form
$e^{-\gamma z^{2}}F(z)$ with $\gamma \geq 0$ and $F(z)$ of genus at most
$1$, and it has no zeros in the open half-plane $\im z < -r$. Its zeros
are symmetric about the real axis, so they lie in $S(r)$, and
$\phi(D)f(z)$ has the required canonical product.

The same estimates justify composition of operators. Let
$\psi_1(z)=\sum_{j=0}^{\infty}b_jz^j\in\LP_1$ and
$\psi_2(z)=\sum_{k=0}^{\infty}c_kz^k=e^{-\gamma_1z^2}\psi_3(z)$, where
$\psi_3(z)\in\LP_1$, and let $f(z)$ be as in
Lemma~\ref{lemma:differentialoperator}. Choose $\delta>\gamma_1$ and
$\eta>\gamma_2$ with $4\delta\eta<1$, and then choose $s>1$ and
$\epsilon>0$ with $s^2(\gamma_1+2\epsilon)<\delta$. Set
$d_m=\sum_{j+k=m}|b_jc_k|$. The bounds
$|\psi_1(z)|\leq C_\epsilon e^{\epsilon|z|^2}$ and
$|\psi_2(z)|\leq C_\epsilon e^{(\gamma_1+\epsilon)|z|^2}$, together with
Cauchy's inequalities on the circle $|z|=sR$, give
\[
  \sum_{m = 0}^{\infty}d_mR^m
  = \sum_{j,k \geq 0}|b_jc_k|R^{j+k}
  \leq C_\epsilon^2 e^{s^2(\gamma_1+2\epsilon)R^2}
  \left(\sum_{j = 0}^{\infty}s^{-j}\right)^2
  \leq C e^{\delta R^2}.
\]
Since $d_m\geq 0$, taking $R=\sqrt{m/(2\delta)}$ shows that
$d_m\leq C(2e\delta/m)^{m/2}$ for $m\geq 1$. Also
$|f(z)|\leq C_\eta e^{\eta|z|^2}$. Applying Cauchy's inequality to $f$ on
a circle of radius $\sqrt{m/(2\eta)}$ about $z$, and using
$m!\leq e\sqrt{m}(m/e)^m$, gives, for $|z|\leq\varrho$ and $m\geq 1$,
\[
  d_m|f^{(m)}(z)| \leq C_\varrho\sqrt{m}\,
  e^{\varrho\sqrt{2\eta m}}(4\delta\eta)^{m/2}.
\]
The series of these bounds converges. Thus
$\sum_{j,k\geq 0}|b_jc_kf^{(j+k)}(z)|$ converges uniformly on compact
sets. Differentiating the series for $\psi_2(D)f(z)$ term by term and
rearranging this absolutely convergent double series gives
\[
  \psi_1(D)\big[\psi_2(D)f(z)\big]
  = \sum_{m = 0}^{\infty}\left(\sum_{j+k = m}b_jc_k\right)f^{(m)}(z)
  = (\psi_1\psi_2)(D)f(z).
\]

The assumption $\gamma_1 \gamma_2 <1/4$ is essential.
Levin~\cite[p.~361]{Levin1980} gives the explicit example
$\phi(z)=e^{-\gamma_1 z^2}$ and $f(z)=e^{-\gamma_2 z^2}$ to show that
$\phi(D)f(z)$ diverges at $z=0$ when $\gamma_1\gamma_2=1/4$.

In the lemma the zeros of $f(z)$ are in the strip $S(r)$ as are the zeros
of $\phi(D)f(z)$. So, $\phi(D)$ is an operator that preserves the strip
$S(r)$ containing the zeros. However, our main interest in this paper is
to study the operators $\phi(D)$ such that the zeros of $\phi(D)f(z)$
belong to a strictly smaller strip $S(r_1)$ where $0 \leq r_1 < r$.

\begin{definition}[Complex zero strip decreasing operator or CZSDO]
  \mbox{} \label{definition:LPtothegammaofr}
  \begin{enumerate}
    \item[(a)] Given a function
      $\phi(z)=e^{-\gamma_1 z^2}\phi_1(z) \in \LP$ where
      $\phi_1(z) \in \LP_1$, we define $\LP^{\gamma_1}(r)$ to be the
      subclass of $\LP(r)$ of functions of the form
      $f(z)=e^{-\gamma_2
    z^2}f_1(z)$ where $f_1(z) \in \LP_1(r)$ and $\gamma_2$
      is any nonnegative real number such that $\gamma_1 \gamma_2 < 1/4$.
    \item[(b)] The linear differential operator $\phi(D)$ in part~(a) is
      called a \textit{complex zero strip decreasing operator} if for each
      $r>0$ there exists a corresponding $r_1$ with $0 \leq r_1 < r$ such
      that $\phi(D)f(z) \in \LP(r_1)$ for all
      $f(z) \in \LP^{\gamma_1}(r)$. For short, we will say $\phi(D)$ is a
      CZSDO.
  \end{enumerate}
\end{definition}

Lemma~\ref{lemma:differentialoperator} implies $\phi(D)f(z) \in \LP(r)$
for all $f(z)\in\LP^{\gamma_1}(r)$, which is why we defined
$\LP^{\gamma_1}(r)$ in Definition~\ref{definition:LPtothegammaofr}(a).

Later in this section and in Sections~\ref{section:examples}
and~\ref{section:futherstudy} we also use the term complex zero strip
decreasing operator, or CZSDO, for other linear operators $T$ with the
analogous property: for each $r > 0$ there exists $r_{1}$ with
$0 \leq r_{1} < r$ such that the zeros of $Tf$ lie in $S(r_{1})$ for every
$f$ in the class on which $T$ acts (real polynomials with zeros in $S(r)$,
or functions in $\LP(r)$).

In this paper we will prove two main theorems.
Theorem~\ref{theorem:MainTheorem:1} provides a sufficient condition for
$\phi(D)$ to be a CZSDO. Theorem~\ref{theorem:MainTheorem:2} gives a
necessary condition for $\phi(D)$ to be a CZSDO.

\begin{theorem} \label{theorem:MainTheorem:1}
  Assume $\phi(z)=e^{-\alpha^2 z^2/2}\phi_1(z)$ where
  $\phi_1(z) \in \LP_1$ and $\alpha>0$. If $f(z) \in \LP^{\alpha^2/2}(r)$,
  then $\phi(D)f(z) \in
\LP(\sqrt{r^2-\alpha^2})$. Therefore $\phi(D)$ is a CZSDO.
\end{theorem}

\begin{theorem} \label{theorem:MainTheorem:2}
  \mbox{}
  \begin{enumerate}
    \item[(a)] If $\phi(z) \in \LP_1$ has order $\rho<1$, then $\phi(D)$
      is \textit{not} a CZSDO.
    \item[(b)] If $\phi(z) \in \LP_1$ has order $\rho=1$ but is of minimal
      type, then $\phi(D)$ is \textit{not} a CZSDO.
    \item[(c)] If $\phi(z)=e^{\alpha z}$ where $\alpha \in \R$ and
      $f(z) \in
    \LP(r)$, then $\phi(D)f(z)=f(z+\alpha)$. Hence
      $\phi(D)$ is not a CZSDO.
  \end{enumerate}
  Therefore, a necessary condition for $\phi(z)\in \LP_1$ to define a
  CZSDO is: If the Weierstrass canonical product for $\phi(z)$ is
  \[
    \phi(z) = c z^m e^{\alpha z} \prod_{n} (1-z/a_n)e^{z/a_n},
  \]
  then the product
  \[
    c z^m \prod_{n} (1-z/a_n)e^{z/a_n},
  \]
  with the term $e^{\alpha z}$ omitted, has order $\rho=1$ and type
  $\sigma>0$, or has order $\rho=1$ and maximal type, or has order
  $\rho>1$.
\end{theorem}

Note that since functions in $\LP_1$ and $\LP_1(r)$ have Weierstrass
canonical products of genus $g=0$ or $g=1$ and since the genus is related
to the order $\rho$ by $g \leq \rho \leq g+1$, the order of any of these
functions satisfies $0 \leq \rho \leq 2$. So, in
Theorem~\ref{theorem:MainTheorem:2} the only relevant orders satisfy
$1 \leq \rho \leq 2$.

Theorems~\ref{theorem:MainTheorem:1} and~\ref{theorem:MainTheorem:2} are
important in the context of the following general problem: If
$\Omega
\subseteq \C$ is a set of particular interest and if
$\pi(\Omega)$ is the class of all (real or complex) univariate polynomials
whose zeros lie only in $\Omega$, then characterize the linear
transformations $T \colon
\pi(\Omega) \rightarrow \pi(\Omega)\cup\{0\}$. Furthermore, if $\pi_n(\Omega)$
is the subclass of polynomials in $\pi(\Omega)$ of degree at most~$n$,
then characterize the linear transformations $T \colon
\pi_n(\Omega) \rightarrow \pi(\Omega)\cup\{0\}$.

Borcea and Br\"and\'en~\cite{BorceaBranden2009c} solved these problems in
the case when $\Omega$ is a line, a circle, a closed half-plane, a closed
disk, or the complement of an open disk. They gave several different types
of descriptions that all linear operators having these zero preserving
properties must satisfy. In a recent preprint, Br\"and\'en and
Chasse~\cite{BrandenChasse2014} classify the zero preserving linear
operators when $\Omega=S(r)$ is a strip in the complex plane.

The case when $\Omega=S(r)$ is a strip is the setting of this paper. An
operator $T$ on real polynomials in $\pi(S(r))$ with the CZSDO property
clearly satisfies $T \colon \pi(\R) \rightarrow \pi(\R)\cup\{0\}$.
However, simple examples show that the converse is false. A full
characterization of CZSDOs in the style of Borcea and
Br\"and\'en~\cite{BorceaBranden2009c} and Br\"and\'en and
Chasse~\cite{BrandenChasse2014} would require some kind of modification to
their descriptions. The linear operators in
Theorem~\ref{theorem:MainTheorem:1} give a large class of explicit
examples of linear operators that preserve the strip $S(r)$, but better
yet, are complex zero strip decreasing operators.
Conjecture~\ref{conjecture:MainConjecture} (stated later in
Section~\ref{section:examples}) proposed a complete classification of
CZSDOs of the form $\phi(D)$ where $\phi(z) \in \LP$, but this conjecture
is false.

Our results are related to but different from those of Craven and
Csordas~\cite{Craven1995} in which they studied linear transformations $T$
on real polynomials $p(x)$ such that the number of complex zeros of
$T[p(x)]$ is less than or equal to the number of complex zeros of $p(x)$.
In their enjoyable survey article~\cite{Craven2004}, Craven and Csordas
explain many interesting results pertaining to operators that preserve
reality of zeros.

The rest of the paper is organized as follows: In
Section~\ref{section:philosophy}, we explain a heuristic to help motivate
the context for this paper. In Section~\ref{section:proof:MainTheorem:1},
we prove Theorem~\ref{theorem:MainTheorem:1}. In
Section~\ref{section:proof:MainTheorem:2}, we prove
Theorem~\ref{theorem:MainTheorem:2}. In Section~\ref{section:examples}, we
give several examples and conjectures. Finally, in
Section~\ref{section:futherstudy}, we suggest questions for further study
on this topic.

\section{Some philosophy and heuristics}  \label{section:philosophy}
Much of the discussion in this paper becomes significantly more intuitive
if one keeps in mind the following fundamental fact:

\begin{theorem}[Gauss-Lucas]
  Every convex set containing all the zeros of a nonconstant polynomial
  also contains all of its critical points.
\end{theorem}
Proofs can be found in many places but we especially like the treatise on
the analytic theory of polynomials by Rahman and
Schmeisser~\cite{RahmanSchmeisser2002}.

This theorem provides a natural strategy for constructing examples of
linear operators with particular zero preserving features. As an example
of this approach, the differential linear operator $\phi(D)$ in
Theorems~\ref{theorem:MainTheorem:1} and~\ref{theorem:MainTheorem:2} is
quite natural as follows: Suppose $f$ is a real polynomial with zeros in
$S(r)$, and let $\alpha$ be real and nonzero. If $I$ is the identity
operator and $D$ is differentiation,
\[
  \Big(I-\frac{D}{\alpha}\Big)f(z) = f(z)-\frac{f'(z)}{\alpha}
  =
  - \frac{e^{\alpha z}}{\alpha} \frac{d}{dz}\left(e^{-\alpha z} f(z)\right) .
\]
The zeros of $f(z)-f'(z)/\alpha$ are those of $\tfrac{d}{dz}\big(e^{-\alpha
z}f(z)\big)$. By
considering the Gauss-Lucas Theorem and the approximation
$e^{-\alpha z} \approx (1-\tfrac{\alpha z}{n})^n$, we see that the zeros
of $f(z)-f'(z)/\alpha$ approximately belong to the convex hull of the
zeros of $f(z)$ and the real number $n/\alpha$ which is the only root of
$(1-\alpha
z/n)^n$. But since the roots of $f(z)$ and the root
$n/\alpha$ are in $S(r)$, this convex hull lies inside the strip $S(r)$ as
well. Taking the limit shows that the roots of $f(z)-f'(z)/\alpha$ belong
to $S(r)$. We conclude that if $\phi(z)=z^m \prod_{k=1}^n (1-z/\alpha_k)$
is a polynomial in which all $\alpha_k$ are real, then the zeros of
$\phi(D)f(z)$ belong to the strip $S(r)$. By taking slightly more care, we
can extend both $\phi$ and $f$ to be entire functions that are
sufficiently nice limits of sequences of polynomials (thus obtaining
Lemma~\ref{lemma:differentialoperator}). Hence, we obtain the case
$\phi \in \LP$ and $f \in \LP(r)$, which is the main focus of this paper.

Applications of this general strategy produce a wide variety of
interesting facts about zeros of polynomials. A good reference is Chapter
5 of~\cite{RahmanSchmeisser2002}.

\section{Proof of Theorem~\ref{theorem:MainTheorem:1}}
\label{section:proof:MainTheorem:1}

Assume $\phi(z)=e^{-\alpha^2 z^2/2} \phi_1(z)$ where $\phi_1(z) \in \LP_1$
and let $f(z) \in \LP^{\alpha^2/2}(r)$ for $r>0$. We will show that
$\exp\big(-\tfrac{\alpha^2 D^2}{2}\big)f(z) \in \LP(\sqrt{r^2-\alpha^2})$.
Then since $\phi_1(D)$ preserves $\LP(\sqrt{r^2-\alpha^2})$, it will
follow that $\phi(D)f(z) \in \LP(\sqrt{r^2-\alpha^2})$.

To understand the effect of the differential operator
$\exp(-\frac{\alpha^2 D^2}{2})$ where $\alpha \geq 0$, we first consider
some simpler exponential operators.

\begin{lemma}[Shifting operator]  \label{lemma:shiftingoperator}
  Let $\beta$ be any complex number and let $f$ be any entire function.
  Then $\exp(\beta D) f(z) = f(z+\beta)$.
\end{lemma}

\begin{proof}
  This is a trivial consequence of the Taylor series expansion of
  $f(z+\beta)$ and the definition of $\exp(\beta D)f(z)$ since
  $\exp(\beta D)
f(z)=\sum_{k=0}^{\infty} \frac{\beta^k f^{(k)}(z)}{k!}=f(z+\beta)$.
\end{proof}

\begin{corollary}
  Let $\alpha,\beta\in \R$, let $D$ denote differentiation, and let $f$ be
  an entire function. Since
  \[
    \cos(\alpha z+\beta) = \frac{1}{2}\big(e^{(\alpha z+\beta)i}+e^{-(\alpha z+\beta)i}\big)
  \]
  and
  \[
    \sin(\alpha z+\beta) = \frac{1}{2i}\big(e^{(\alpha z+\beta)i}-e^{-(\alpha z+\beta)i}\big),
  \]
  it follows immediately that
  \[
    \cos(\alpha D+\beta)f(z) = \tfrac{1}{2}\big(e^{i\beta}f(z+i\alpha)+e^{-i\beta}f(z-i\alpha)\big)
  \]
  and
  \[
    \sin(\alpha D+\beta)f(z) = \tfrac{1}{2i}\big(e^{i\beta}f(z+i\alpha)-e^{-i\beta}f(z-i\alpha)\big).
  \]
\end{corollary}

\begin{lemma}[Effect of cosine and sine operators] \label{lemma:cosineandsineoperators}
  Let $\alpha,\beta\in \R$ where $\alpha>0$, let $D$ denote
  differentiation, and let $f(z)\in \LP(r)$ where $r \geq 0$. Then
  \[
    \cos(\alpha D+\beta)f(z) \in \LP(\sqrt{r^2-\alpha^2})
  \]
  and
  \[
    \sin(\alpha D+\beta)f(z) \in \LP(\sqrt{r^2-\alpha^2}).
  \]
\end{lemma}

\begin{proof}
  One can compare the Weierstrass product representation of $f(z+i\alpha)$
  with that of $f(z-i\alpha)$ to conclude that the zeros of
  $\cos(\alpha D+\beta)f(z)$ and $\sin(\alpha D+\beta)f(z)$ are in the
  strip $S(\sqrt{r^2-\alpha^2})$. P\'olya used this idea in the proof of
  Hilfssatz II in his 1926 paper~\cite{Polya1926} on the Riemann zeta
  function in which he proved a Riemann hypothesis for a `fake' zeta
  function. He considered a slightly simpler case, but likely was aware of
  the fact stated in this lemma. One proof attributed to de Bruijn when
  $f(z)$ is a polynomial is found in~\cite[Theorem 2.5.1,
  p.~88]{RahmanSchmeisser2002}. This lemma may be regarded as an extension
  of Jensen's theorem: If $f$ is a polynomial with real coefficients, then
  the nonreal critical points of $f$ lie in the union of all the Jensen
  discs of $f$. For additional history and various generalizations see
  Section 2.4 of~\cite{RahmanSchmeisser2002}.

  By Lemma~\ref{lemma:differentialoperator}, both images belong to
  $\LP(r)$. It remains to locate their zeros more precisely.

  If $f(z)$ is of the form $f(z)=ce^{\delta z}$ where $c,\delta \in \R$,
  then $f(z) \in
\LP$ and hence
  \[
    \cos(\alpha D+\beta)f(z) \in \LP \subseteq \LP(\sqrt{r^2-\alpha^2}).
  \]
  Similarly,
  \[
    \sin(\alpha D+\beta)f(z) \in \LP \subseteq \LP(\sqrt{r^2-\alpha^2}).
  \]
  The lemma is true in this case.

  If $f(z)$ is not of the form $ce^{\delta z}$, then the Weierstrass
  canonical product for $f(z)$ contains a term of the form
  $e^{-\gamma z^2}$ where $\gamma >0$ or the product contains at least one
  term corresponding to a root of $f(z)$. Denote the nonzero real zeros of
  $f(z)$ by $r_n$ and denote the complex roots with positive imaginary
  part by $s_n+it_n$. By combining terms for complex conjugate roots in
  the Weierstrass product, we find that $f(z)$ has the form
  \begin{align} \label{eqn:cosine:0}
     & c z^m e^{\delta z-\gamma z^2}
    \prod_n\big(1-\tfrac{z}{r_n}\big) e^{z/r_n}
    \prod_n \big(1-\tfrac{z}{s_n+it_n}\big)
    \big(1-\tfrac{z}{s_n-it_n}\big)
    \exp\big(\tfrac{2s_n z}{s_n^2+t_n^2}\big) \\
     & = c z^m e^{\delta z-\gamma z^2}
    \prod_n\big(1-\tfrac{z}{r_n}\big) e^{z/r_n}
    \prod_n \tfrac{(s_n+it_n-z)(s_n-it_n -z)}{s_n^2+t_n^2}
    \exp\big(\tfrac{2s_n z}{s_n^2+t_n^2}\big) \nonumber
  \end{align}
  where $\delta ,\gamma ,c \in \R$, $\gamma  \geq 0$, and $m$ is a
  nonnegative integer. For $z=x+iy$,
  \begin{align} \label{eqn:cosine:1}
    |f(z)|^2 & = |c|^2 (x^2+y^2)^m e^{2\delta x+2\gamma (y^2-x^2)} \prod_n \tfrac{(x-r_n)^2+y^2}{r_n^2} \exp\big(\tfrac{2x}{r_n}\big) \\
             & \qquad \times \prod_n
    \tfrac{(x-s_n)^2+(y-t_n)^2}{s_n^2+t_n^2} \cdot
    \tfrac{(x-s_n)^2+(y+t_n)^2}{s_n^2+t_n^2} \cdot
    \exp\big(\tfrac{4s_n x}{s_n^2+t_n^2}\big). \nonumber
  \end{align}

  Let $z=x+iy$ be a root of
  \[
    \cos(\alpha D+\beta)f(z) = \tfrac{1}{2}\big(e^{i\beta}f(z+i\alpha)+e^{-i\beta}f(z-i\alpha)\big)
  \]
  with $y \geq 0$. Then
  \[
    |f(z-i\alpha)|^2 = |f(z+i\alpha)|^2.
  \]
  By way of contradiction, assume that the root $z$ is not in the strip
  $S(\sqrt{r^2-\alpha^2})$. Then $y>0$ and $y^2>r^2-\alpha^2$. We will
  show that each nonconstant term in the product for $|f(z-i\alpha)|^2$ is
  less than or equal to the corresponding term in the product for
  $|f(z+i\alpha)|^2$ and that strict inequality holds for at least one
  term. This will show that $|f(z-i\alpha)|^2<|f(z+i\alpha)|^2$ contrary
  to the hypothesis.

  First, we consider the factors of $|f(z-i\alpha)|^2$ and
  $|f(z+i\alpha)|^2$ associated with the exponential term
  $e^{2\delta x+2\gamma (y^2-x^2)}$ in Eq.~\eqref{eqn:cosine:1}. Since
  $y>0$ this gives
  \begin{equation}\label{eqn:cosine:2}
    e^{2\delta x+2\gamma ((y-\alpha)^2-x^2)} \leq e^{2\delta x+2\gamma ((y+\alpha)^2-x^2)},
  \end{equation}
  where the inequality is strict if and only if $\gamma >0$.

  Next, we consider the factors associated with real roots (if there are
  any) of $f(z)$. Since $y>0$, we have
  \begin{equation} \label{eqn:cosine:3}
    (x^2+(y-\alpha)^2)^m < (x^2+(y+\alpha)^2)^m
  \end{equation}
  when $m>0$, and
  \begin{equation} \label{eqn:cosine:4}
    \frac{(x-r_n)^2+(y-\alpha)^2}{r_n^2} \exp\left(\frac{2x}{r_n}\right)
    <
    \frac{(x-r_n)^2+(y+\alpha)^2}{r_n^2} \exp\left(\frac{2x}{r_n}\right).
  \end{equation}

  Finally, we consider the factors of $|f(z-i\alpha)|^2$ and
  $|f(z+i\alpha)|^2$ associated with complex conjugate pairs of roots (if
  there are any) of $f(z)$. We will show that
  \begin{align} \label{eqn:cosine:5}
     & \frac{(x-s_n)^2+(y-\alpha-t_n)^2}{s_n^2+t_n^2} \cdot
    \frac{(x-s_n)^2+(y-\alpha+t_n)^2}{s_n^2+t_n^2} \cdot
    \exp\left(\frac{4s_n x}{s_n^2+t_n^2}\right) \\
     & < \frac{(x-s_n)^2+(y+\alpha-t_n)^2}{s_n^2+t_n^2} \cdot
    \frac{(x-s_n)^2+(y+\alpha+t_n)^2}{s_n^2+t_n^2} \cdot
    \exp\left(\frac{4s_n x}{s_n^2+t_n^2}\right).\nonumber
  \end{align}
  Inequality~\eqref{eqn:cosine:5} holds if and only if
  \begin{align} \label{eqn:cosine:6}
     & [(x-s_n)^2+(y-\alpha-t_n)^2][(x-s_n)^2+(y-\alpha+t_n)^2] \\
     & < [(x-s_n)^2+(y+\alpha-t_n)^2][(x-s_n)^2+(y+\alpha+t_n)^2].  \nonumber
  \end{align}
  Subtracting the left hand side of~\eqref{eqn:cosine:6} from the right
  hand side along with a small calculation gives
  \begin{align*}
     & [(x-s_n)^2+(y+\alpha-t_n)^2][(x-s_n)^2+(y+\alpha+t_n)^2] \\
     & \quad - [(x-s_n)^2+(y-\alpha-t_n)^2][(x-s_n)^2+(y-\alpha+t_n)^2] \\
     & = 8\alpha y\big((x-s_n)^2+y^2-t_n^2+\alpha^2\big).
  \end{align*}
  Thus~\eqref{eqn:cosine:5} and~\eqref{eqn:cosine:6} hold if and only if
  \begin{equation}\label{eqn:cosine:7}
    (x-s_n)^2+y^2 > t_n^2-\alpha^2.
  \end{equation}
  Because the root $z$ does not belong to $S(\sqrt{r^2-\alpha^2})$ but
  does satisfy $0 \leq \im z \leq r$, it follows that
  $y^2>r^2-\alpha^2 \geq
t_n^2-\alpha^2$ and therefore~\eqref{eqn:cosine:7} holds
  implying that inequality~\eqref{eqn:cosine:5} holds.

  If $\gamma >0$ in the product representation of $f(z)$ in
  Eq.~\eqref{eqn:cosine:0}, then inequality~\eqref{eqn:cosine:2} is
  strict. If $f(z)$ has at least one root then at least one of strict
  inequalities~\eqref{eqn:cosine:3}, \eqref{eqn:cosine:4}, or
  \eqref{eqn:cosine:5} holds. Since $(y+\alpha)^2>r^2$, we have
  $f(z+i\alpha)\neq 0$. Divide each finite partial product for
  $|f(z-i\alpha)|^2$ by the corresponding partial product for
  $|f(z+i\alpha)|^2$. Once a fixed strict factor is included, these ratios
  are bounded above by a fixed number less than $1$. Passing to the limit
  therefore preserves strict inequality. Either way,
  \[
    |f(z-i\alpha)|^2 < |f(z+i\alpha)|^2,
  \]
  which is a contradiction. This completes the proof of
  Lemma~\ref{lemma:cosineandsineoperators} for the operator
  $\cos(\alpha D+\beta)$. The proof for the operator
  $\sin(\alpha D+\beta)$ is entirely similar. Zeros in the lower
  half-plane are handled by complex conjugation. The proof of the lemma is
  complete.
\end{proof}

Recently Lagarias~\cite{Lagarias2005} applied operators such as the one in
Lemma~\ref{lemma:cosineandsineoperators} to study the zero distribution on
the `critical line' for various differenced $L$-functions from analytic
number theory. Several generalizations of
Lemma~\ref{lemma:cosineandsineoperators} can also be found in
Cardon~\cite{Cardon2002,Cardon2004,Cardon2005} and Cardon and
Nielsen~\cite{CardonNielsen2002}.

\begin{lemma}\label{lemma:exponentialsquared}
  Let $\alpha \geq 0$ and assume $f(z) \in \LP^{\alpha^2/2}(r)$. Then
  \[
    \exp(-\tfrac{\alpha^2 D^2}{2}) f(z) \in \LP(\sqrt{r^2-\alpha^2}).
  \]
\end{lemma}

\begin{proof}
  If $\alpha=0$ the assertion is immediate, so assume $\alpha>0$. We take
  advantage of the limit formula
  \[
    \lim_{n \rightarrow \infty} \left(\cos\left(\tfrac{\alpha z}{n}\right)\right)^{n^2} = \exp(-\tfrac{\alpha^2 z^2}{2}).
  \]
  Initially, let $f(z)$ be a polynomial in $\LP(r)$. Applying the formula
  in Lemma~\ref{lemma:cosineandsineoperators} a total of $n^2$ times for
  the operator $\cos(\tfrac{\alpha D}{n})$ gives
  \[
    \left(\cos\left(\tfrac{\alpha D}{n}\right)\right)^{n^2} f(z) \in \LP\left(\sqrt{r^2 - \tfrac{\alpha^2}{n^2}-\cdots-\tfrac{\alpha^2}{n^2}}\right) = \LP(\sqrt{r^2 - \alpha^2}).
  \]
  If an intermediate operator produces only real zeros, the subsequent
  cosine operators preserve that property. Taking the limit gives
  $\exp(-\tfrac{\alpha^2 D^2}{2})f(z) \in
\LP(\sqrt{r^2-\alpha^2})$ for polynomials.

  For the passage to entire functions, write
  \[
    \left(\cos\left(\tfrac{\alpha z}{n}\right)\right)^{n^2}
    = \sum_{k = 0}^{\infty}(-1)^k c_{n,k}z^{2k}.
  \]
  The coefficients of $\cosh(\alpha z/n)$ are nonnegative and are at most
  the corresponding coefficients of $\exp(\alpha^2 z^2/(2n^2))$, since
  $(2k)!\geq 2^k k!$. Multiplying $n^2$ copies gives
  \[
    0 \leq c_{n,k} \leq \frac{(\alpha^2/2)^k}{k!},
    \qquad
    \lim_{n\rightarrow\infty}c_{n,k} = \frac{(\alpha^2/2)^k}{k!}.
  \]
  If $f(z)=e^{-\beta^2 z^2/2}f_1(z)$ with $\beta\geq 0$, $f_1\in\LP_1(r)$,
  and $\alpha\beta<1$, the convergence estimates for
  Lemma~\ref{lemma:differentialoperator} show that
  \[
    \sum_{k = 0}^{\infty}\frac{(\alpha^2/2)^k}{k!}|f^{(2k)}(z)|
  \]
  converges uniformly on compact sets. Thus the same operator limit holds
  uniformly on compact sets for this $f$. By the composition identity
  following Lemma~\ref{lemma:differentialoperator}, applying the operator
  $\big(\cos(\tfrac{\alpha D}{n})\big)^{n^{2}}$ to $f$ is the same as
  applying $\cos(\tfrac{\alpha D}{n})$ to $f$ a total of $n^{2}$ times, so
  each of the approximating operators narrows the strip as above.
  Hurwitz's theorem gives the asserted zero location in the limit, unless
  the limit is identically zero, and
  Lemma~\ref{lemma:differentialoperator} gives its membership in $\LP(r)$.
  This proves the lemma on the stated domain.
\end{proof}

Now that we understand the effect of the operator
$\exp(-\tfrac{\alpha^2
D^2}{2})$, we can finish the proof of
Theorem~\ref{theorem:MainTheorem:1}. Assume
$\phi(z)=e^{-\alpha^2 z^2/2} \phi_1(z)$ where $\phi_1(z) \in \LP_1$ and
let $f(z) \in \LP^{\alpha^2/2}(r)$. By
Lemma~\ref{lemma:exponentialsquared},
\[
  \exp\big(-\tfrac{\alpha^2 D^2}{2}\big)f(z) \in \LP(\sqrt{r^2-\alpha^2}).
\]
By Lemma~\ref{lemma:differentialoperator}, $\phi_1(D)$ maps
$\LP(\sqrt{r^2-\alpha^2})$ into itself. Hence, by the composition identity
following Lemma~\ref{lemma:differentialoperator},
\[
  \phi(D)f(z) = \phi_1(D)\exp\big(-\tfrac{\alpha^2 D^2}{2}\big)f(z) \in
  \LP(\sqrt{r^2-\alpha^2}).
\]
Since $\alpha>0$ this proves that $\phi(D)$ is a CZSDO and the proof of
Theorem~\ref{theorem:MainTheorem:1} is complete.

\section{Proof of Theorem~\ref{theorem:MainTheorem:2}}
\label{section:proof:MainTheorem:2} Since the proof requires the concepts of
order~$\rho$ and type~$\sigma$ of an entire function, we recall their
definitions, which we take from Chapter~1 of Levin~\cite{Levin1980}. In these definitions the symbol $r$ represents the radius of a circle in the complex plane rather than the half-width of a strip. For an
arbitrary entire function $\phi$, set
\[
  M_{\phi}(r) = \max_{|z| = r} |\phi(z)|.
\]
The function $\phi$ is said to have \textit{finite order} if there exists
a positive real number $k>0$ such that
\begin{equation} \label{inequality:order:1}
  M_{\phi}(r) < e^{r^k}
\end{equation}
for all sufficiently large $r$. If $\phi$ has finite order, the
\textit{order} $\rho$ of $\phi$ is defined to be the greatest lower bound
of the numbers $k$ in~\eqref{inequality:order:1}.

For nonconstant $\phi$, it follows that for arbitrary $\epsilon>0$
\begin{equation} \label{inequality:order:2}
  e^{r^{\rho-\epsilon}} < M_{\phi}(r) < e^{r^{\rho+\epsilon}}
\end{equation}
where the inequality on the right holds for all sufficiently large $r$ and
the inequality on the left holds for some positive increasing sequence
$\{r_n\}$ with $\lim_{n \rightarrow\infty} r_n=\infty$.

The \textit{type} $\sigma$ of a function $\phi$ having positive finite
order $\rho$ is the greatest lower bound of the positive numbers
$\epsilon$ such that
\begin{equation}\label{inequality:type}
  M_{\phi}(r) < e^{\epsilon r^{\rho}}
\end{equation}
for all sufficiently large $r$. If $\sigma=0$, then $\phi$ is said to have
\textit{minimal type}.

\begin{lemma}  \label{lemma:evenphinotCZSDO}
  Let $\phi(z) \in \LP_1$. If $\phi(D)\phi(-D)$ is not a CZSDO, then
  $\phi(D)$ is not a CZSDO. Therefore, in proving
  Theorem~\ref{theorem:MainTheorem:2} there is no loss of generality in
  assuming that $\phi(z) \in \LP_1$ is an even function.
\end{lemma}

\begin{proof}
  Suppose, by way of contradiction, that $\phi(D)$ is a CZSDO. For any
  $r>0$, there exists $r_1$ with $0 \leq r_1 < r$ such that for any
  $f(z) \in \LP(r)$ it follows that $\phi(D) f(z) \in \LP(r_1)$. The
  operator $\phi(-D)$ maps $\LP(r)$ into itself. So,
  $\phi(D)\phi(-D)f(z) \in \LP(r_1)$. Let $\psi(z) = \phi(z)\phi(-z)$. By
  the composition identity following
  Lemma~\ref{lemma:differentialoperator},
  $\phi(D)\phi(-D)f(z) = \psi(D)f(z)$, implying that
  $\phi(D)\phi(-D) = \psi(D)$ is a CZSDO, which contradicts the
  hypothesis. Therefore, $\phi(D)$ is not a CZSDO. The function $\psi(z)$
  is even and belongs to $\LP_1$, and $M_{\psi}(r) \leq M_{\phi}(r)^{2}$.
  Consequently, an order less than $1$ remains less than $1$, and order
  $1$ of minimal type becomes either order less than $1$ or order $1$ of
  minimal type. This justifies the reduction in parts~(a) and~(b).
\end{proof}

By Lemma~\ref{lemma:evenphinotCZSDO}, there is no loss of generality in
assuming $\phi(z)$ is an even function. Therefore, we will assume
$\phi(z)$ is even in the rest of the proof of
Theorem~\ref{theorem:MainTheorem:2}.

The key to proving Theorem~\ref{theorem:MainTheorem:2} will be to let
$\phi(D)$ act on extremal example functions with evenly spaced zeros that
are on the boundary of the strip $S(r)$. If $\phi(0)\neq 0$, we will
consider $\phi(D)f_a(z)$, where
\begin{equation}  \label{eqn:fsuba}
  f_a(z) = \cos(a(z-ir))\cos(a(z+ir)) = \tfrac{1}{2}\big(\cos(2az)+\cosh(2ar)\big).
\end{equation}
If $\phi(0)=0$, we will consider $\phi(D)g_a(z)$, where
\begin{align} \label{eqn:gsuba}
  g_a(z) & = \cos^2\big(a(z-ir)\big)\cos^2\big(a(z+ir)\big).
\end{align}
In both cases, we'll show $\phi(D)$ is not a CZSDO by choosing $a>0$ to be
sufficiently large. Computing with $f_a(z)$ and $g_a(z)$ will require the
following easy lemma:

\begin{lemma}  \label{lemma:eigenfunction}
  For any $\phi(z) \in \LP$, $\phi(D)e^{a z} = \phi(a) e^{a z}$.
  Consequently, if $\phi(z)$ is even, then
  \begin{align*}
    \phi(D) \cos(a z) & = \phi(ai)\cos(az) \\
    \phi(D) \sin(a z) & = \phi(ai)\sin(az) \\
    \phi(D) \cosh(az) & = \phi(a) \cosh(az) \\
    \phi(D) \sinh(az) & = \phi(a) \sinh(az).
  \end{align*}
\end{lemma}
\begin{proof}
  Write $\phi(z)=\sum_{n=0}^{\infty} c_n z^n$. Then
  \[
    \phi(D) e^{a z} = \sum_{n = 0}^{\infty} c_n \tfrac{d^n}{dz^n} e^{a z}
    = \sum_{n = 0}^{\infty} c_n a^n e^{a z} = \phi(a) e^{a z}.
  \]
  When $\phi(z)$ is even, the formulas involving $\cos(az)$, $\sin(az)$,
  $\cosh(az)$, and $\sinh(az)$ follow immediately by expressing them in
  terms of exponential functions.
\end{proof}

\begin{lemma} \label{lemma:phi:nonvanishingat0}
  Let $\phi(z) \in \LP_1$ be an even function and suppose
  $\phi(0) \neq 0$.
  \begin{enumerate}
    \item[(a)] If $\phi(z)$ has order $\rho<1$, then $\phi(D)$ is
      \textit{not} a CZSDO.
    \item[(b)] If $\phi(z)$ has order $\rho=1$ but is of minimal type,
      then $\phi(D)$ is \textit{not} a CZSDO.
  \end{enumerate}
\end{lemma}

\begin{proof}
  By multiplying $\phi(z)$ by a nonzero real number, if necessary, we may
  suppose without loss of generality that $\phi(0)=1$. If $\phi\equiv 1$,
  then $\phi(D)$ is the identity and is not a CZSDO, as is seen by
  applying it to $z^2+r^2$. We may therefore assume that $\phi$ is
  nonconstant. Because $\phi\in\LP_1$ is even and $\phi(0)=1$, its
  canonical product consists of factors $1-z^2/\alpha_n^2$ with real
  $\alpha_n\neq 0$. Hence,
  \[
    M_\phi(r) = \max_{|z| = r} |\phi(z)| = \phi(ir).
  \]
  Also since $\phi(0)=1$ if $\beta$ is constant, then
  $\phi(D)\beta = \beta$.

  Applying the operator $\phi(D)$ to the function $f_a(z)$ in
  Eq.~\eqref{eqn:fsuba} and using Lemma~\ref{lemma:eigenfunction} gives
  \begin{align*}
    \phi(D) f_a(z) & = 2^{-1}\big( \phi(2ai)\cos(2az)+ \cosh(2ar)\big) \\
                   & = 2^{-1} \phi(2ai) \left( \cos(2az) + \frac{\cosh(2ar)}{\phi(2ai)}\right).
  \end{align*}
  Since $a$ is positive $\phi(2ai)>1$. There are two possible cases:
  \begin{enumerate}
    \item[(i)] If $0<\tfrac{\cosh(2ar)}{\phi(2ai)} \leq 1$, the zeros of
      $\phi(D)f_a(z)$ are real.
    \item[(ii)] If $\tfrac{\cosh(2ar)}{\phi(2ai)} >1$, the zeros of
      $\phi(D)f_a(z)$ are complex.
  \end{enumerate}

  In case~(i), there exists $r_1$ with $0 \leq 2ar_1 \leq \pi/2$ such that
  \[
    \cos(2ar_1) = \frac{\cosh(2ar)}{\phi(2ai)}.
  \]
  Then
  \begin{align*}
    \phi(D)f_a(z) & = 2^{-1} \phi(2ai) \big( \cos(2az)+\cos(2ar_1) \big) \\
                  & = \phi(2ai) \cos(a(z-r_1)) \cos (a(z+r_1)).
  \end{align*}
  This verifies that the roots of $\phi(D)f_a(z)$ are real in the
  case~(i), as claimed.

  In case~(ii), since $\cosh(2ar)/\phi(2ai)>1$ and $\phi(2ai)>1$, there
  exists $r_1$ with $0<r_1<r$ such that
  \[
    1 < \cosh(2ar_1) = \frac{\cosh(2ar)}{\phi(2ai)} < \cosh(2ar).
  \]
  We obtain
  \begin{align*}
    \phi(D)f_a(z) & = 2^{-1}\phi(2ai)\big( \cos(2az) + \cosh(2ar_1) \big) \\
                  & = \phi(2ai) \cos(a(z-ir_1)) \cos(a(z+i r_1)),
  \end{align*}
  which shows that the roots of $\phi(D)f_a(z)$ are complex with imaginary
  part $\pm r_1$. Solving for $r_1$ gives
  \begin{equation} \label{equation:r1:phinonvanishing}
    r_1 = \frac{1}{2a} \cosh^{-1}\left(\frac{\cosh(2ar)}{\phi(2ai)}\right).
  \end{equation}

  We will show that, by choosing~$a$ to be sufficiently large, case~(ii)
  occurs and $r_1$ can be made to be arbitrarily close to $r$, proving
  that $\phi(D)$ is not a CZSDO.

  As in part~(a) of Lemma~\ref{lemma:phi:nonvanishingat0}, suppose
  $\phi(z)$ has order $\rho<1$. Choose $\epsilon>0$ with
  $\rho+\epsilon<1$. For all sufficiently large positive~$a$,
  \[
    M_{\phi}(2a) = \phi(2ai) < e^{(2a)^{\rho+\epsilon}}.
  \]
  Then for sufficiently large $a$ we have
  \[
    \frac{\cosh(2ar)}{\phi(2ai)} > \frac{e^{2ar}}{2 e^{(2a)^{\rho+\epsilon}}} > 1.
  \]
  Therefore, the roots of $\phi(D)f_a(z)$ have nonzero imaginary part
  $\pm r_1$, as in Eq.~\eqref{equation:r1:phinonvanishing}, and
  \begin{align*}
    r_1 & = \frac{1}{2a} \cosh^{-1}\left(\frac{\cosh(2ar)}{\phi(2ai)}\right) \\
        & > \frac{1}{2a} \cosh^{-1}\left(\frac{e^{2ar}}{2 e^{(2a)^{\rho+\epsilon}}}\right) \\
        & > \frac{1}{2a} \log \left(\frac{e^{2ar}}{2 e^{(2a)^{\rho+\epsilon}}}\right) \\
        & = r-(2a)^{\rho+\epsilon-1} - \frac{\log(2)}{2a}.
  \end{align*}
  By choosing $a$ to be sufficiently large, this lower bound on $r_1$ can
  be made to be arbitrarily close to $r$. Therefore, there does not exist
  $r'$ with $0\leq r' < r$ such that $\phi(D)h(z) \in \LP(r')$ for all
  $h(z) \in \LP(r)$. This shows that $\phi(D)$ is not a CZSDO when
  $\rho<1$, proving Lemma~\ref{lemma:phi:nonvanishingat0}(a).

  Next, as in part~(b) of Lemma~\ref{lemma:phi:nonvanishingat0}, assume
  $\phi(z)$ has order $\rho=1$ and has minimal type. For any $\epsilon>0$
  and all sufficiently large $t$,
  \[
    M_\phi(t) = \phi(it) < e^{\epsilon t}.
  \]
  Choosing $\epsilon$ to be very small relative to $r$ and letting $a$ be
  sufficiently large gives
  \[
    \frac{\cosh(2ar)}{\phi(2ai)} > \frac{e^{2ar}}{2\phi(2ai)} > \frac{e^{2ar}}{2 e^{2\epsilon a}} > 1.
  \]
  Therefore, the roots of $\phi(D)f_a(z)$ have imaginary part $\pm r_1$,
  as in Eq.~\eqref{equation:r1:phinonvanishing}, and
  \begin{align*}
    r_1 & = \frac{1}{2a} \cosh^{-1}\left(\frac{\cosh(2ar)}{\phi(2ai)}\right) \\
        & > \frac{1}{2a} \cosh^{-1}\left(\frac{e^{2ar}}{2 e^{2\epsilon a}}\right) \\
        & > \frac{1}{2a} \log \left(\frac{e^{2ar}}{2 e^{2\epsilon a}}\right) \\
        & = r - \epsilon - \frac{\log 2}{2a}.
  \end{align*}
  By choosing $\epsilon$ to be sufficiently small and $a$ sufficiently
  large, this lower bound for $r_1$ can be made to be arbitrarily close to
  $r$. This shows that $\phi(D)$ is not a CZSDO when $\phi(z)$ has order
  $\rho=1$ and has minimal type, proving
  Lemma~\ref{lemma:phi:nonvanishingat0}(b).
\end{proof}

Combining Lemmas~\ref{lemma:evenphinotCZSDO}
and~\ref{lemma:phi:nonvanishingat0} proves
Theorem~\ref{theorem:MainTheorem:2} parts~(a) and~(b) in the case
$\phi(0)
\neq 0$. We next deal with the case $\phi(0)=0$.

\begin{lemma} \label{lemma:phi:vanishingat0}
  Let $\phi(z) \in \LP_1$ be an even function and suppose $\phi(0) = 0$.
  \begin{enumerate}
    \item[(a)] If $\phi(z)$ has order $\rho<1$, then $\phi(D)$ is
      \textit{not} a CZSDO.
    \item[(b)] If $\phi(z)$ has order $\rho=1$ but is of minimal type,
      then $\phi(D)$ is \textit{not} a CZSDO.
  \end{enumerate}
\end{lemma}

\begin{proof}
  After multiplying $\phi(z)$ by a constant, if necessary, we may assume
  $\phi(z)$ is of the form
  \[
    \phi(z) = z^{2m} \phi_1(z)
  \]
  where $m$ is a positive integer and $\phi_1\in\LP_1$ is even with
  $\phi_1(0)=1$. We will let $\phi(D)$ act on the function $g_a(z)$ from
  Eq.~\eqref{eqn:gsuba}. Rewrite $g_a(z)$ as
  \begin{align*}
    g_a(z) & = \cos^2(a(z-ir))\cos^2(a(z+ir)) \\
           & = \frac{1}{4} + \frac{\cosh(4ar)}{8} + \frac{\cos(4az)}{8} + \frac{\cosh(2ar)\cos(2az)}{2}.
  \end{align*}
  Differentiating $2m$-times gives
  \[
    g_a^{(2m)}(z)
    = (-1)^m \left(\frac{(4a)^{2m} \cos(4az)}{8} + \frac{(2a)^{2m} \cosh(2ar)\cos(2az)}{2}\right).
  \]
  Applying $\phi_1(D)$ to this expression with the help of
  Lemma~\ref{lemma:eigenfunction} gives
  \begin{align*}
     & \phi(D)g_a(z) \\
     & =
    (-1)^{m} \left(
             \frac{(4a)^{2m} \phi_1(4ai) \cos(4az)}{8}
    +
    \frac{(2a)^{2m} \cosh(2ar)\phi_1(2ai)\cos(2az)}{2}
    \right) \\
     & =
    (-1)^m 2^{2m-1} a^{2m} \big(
                           4^{m-1} \phi_1(4ai) \cos(4az) + \cosh(2ar)\phi_1(2ai) \cos(2az)\big).
  \end{align*}
  Let $z$ be a zero of $\phi(D)g_a(z)$. Then
  \begin{align*}
    0 & = 2^{2m-2} \phi_1(4ai) \cos(4az) + \cosh(2ar)\phi_1(2ai) \cos(2az) \\
      & = 2^{2m-2} \phi_1(4ai) \big(2 \cos^2(2az)-1\big) + \cosh(2ar)\phi_1(2ai) \cos(2az).
  \end{align*}
  Solving for $\cos(2az)$ results in
  \begin{align*}
    \cos(2az) & =
    \tfrac{-\cosh(2ar)\phi_1(2ai) \pm \sqrt{\cosh^2(2ar)\phi_1^2(2ai) + 2^{4m-1}\phi_1^2(4ai)}}{2^{2m}\phi_1(4ai)}.
  \end{align*}
  There are two types of solutions depending on the choice of sign. It
  will suffice for us to consider only those solutions corresponding to
  choosing the negative sign. For these solutions,
  \begin{equation} \label{inequality:phivanishingat0:1}
    0 = \cos(2az) +
    \tfrac{\cosh(2ar)\phi_1(2ai) + \sqrt{\cosh^2(2ar)\phi_1^2(2ai)
        + 2^{4m-1}\phi_1^2(4ai)}}{2^{2m}\phi_1(4ai)}.
  \end{equation}
  The fraction expression in~\eqref{inequality:phivanishingat0:1} has the
  lower bound:
  \begin{align}  \label{inequality:phivanishingat0:2}
     & \tfrac{\cosh(2ar)\phi_1(2ai) + \sqrt{\cosh^2(2ar)\phi_1^2(2ai)
           + 2^{4m-1}\phi_1^2(4ai)}}{2^{2m}\phi_1(4ai)}
    > \frac{\cosh(2ar) \phi_1(2ai)}{2^{2m-1} \phi_1(4ai)}.
  \end{align}

  As in Lemma~\ref{lemma:phi:vanishingat0}(a), assume $\phi(z)$ has order
  $\rho<1$. The order of $\phi_1(z)$ is the same as that of $\phi(z)$, and
  $\phi_1(it)\geq 1$ for $t\geq 0$. Choose $\epsilon>0$ with
  $\rho+\epsilon<1$, and let $\{a_n\}$ be any positive increasing sequence
  tending to $\infty$. For all sufficiently large $a_n$,
  \[
    M_{\phi_1}(4a_n) = \phi_1(4a_n i) < e^{(4a_n)^{\rho+\epsilon}}.
  \]
  Letting $a_n$ be sufficiently large gives
  \begin{align}  \label{inequality:phivanishingat0:3}
    \frac{\cosh(2a_n r)\phi_1(2 a_n i)}{2^{2m-1} \phi_1(4a_n i)}
     & > \frac{e^{2a_n r}}{2^{2m} e^{(4a_n)^{\rho+\epsilon}}} > 1.
  \end{align}
  This implies that the fraction term in
  Eq.~\eqref{inequality:phivanishingat0:1} is larger than $1$. Hence,
  there exists $r_1$ with $0<r_1\leq r$ (by
  Lemma~\ref{lemma:differentialoperator}) such that
  \begin{align*}
    0 & = \cos(2 a_n z)+ \tfrac{\cosh(2a_nr)\phi_1(2a_ni) + \sqrt{\cosh^2(2a_nr)\phi_1^2(2a_ni)
                             + 2^{4m-1}\phi_1^2(4a_ni)}}{2^{2m}\phi_1(4a_ni)} \\
      & = \cos(2a_n z)+\cosh(2a_n r_1) \\
      & = 2 \cos(a_n(z-ir_1))\cos(a_n(z+ir_1)).
  \end{align*}
  Thus $\phi(D) g_{a_n}(z)$ has roots with imaginary part $r_1$. Solving
  for $r_1$ and using inequalities~\eqref{inequality:phivanishingat0:2}
  and~\eqref{inequality:phivanishingat0:3} gives a lower bound for $r_1$.
  \begin{align*}
    r_1 & = \frac{1}{2a_n} \cosh^{-1}\left(\tfrac{\cosh(2a_nr)\phi_1(2a_ni) + \sqrt{\cosh^2(2a_nr)\phi_1^2(2a_ni)
                                               + 2^{4m-1}\phi_1^2(4a_ni)}}{2^{2m}\phi_1(4a_ni)}\right) \\
        & > \frac{1}{2a_n} \cosh^{-1}\left(\frac{e^{2a_n r}}{2^{2m} e^{(4a_n)^{\rho+\epsilon}}}\right) \\
        & > \frac{1}{2a_n} \log\left(\frac{e^{2a_n r}}{2^{2m} e^{(4a_n)^{\rho+\epsilon}}}\right) \\
        & = r- \frac{1}{2a_n} \left( (4a_n)^{\rho+\epsilon}+\log(2^{2m})\right) .
  \end{align*}
  By choosing $a_n$ sufficiently large, we cause the lower bound on $r_1$
  to be arbitrarily close to $r$. So, $\phi(D)$ is not a CZSDO in the case
  $\rho<1$, proving Lemma~\ref{lemma:phi:vanishingat0}(a).

  Next, as in Lemma~\ref{lemma:phi:vanishingat0}(b), assume $\phi(z)$ has
  order $\rho=1$ and minimal type. Then the same is true for $\phi_1(z)$.
  We choose a small positive $\epsilon$ with $0<\epsilon<\min\{1,r/2\}$.
  Thus there exists a positive increasing sequence $\{a_n\}$ tending to
  $\infty$ such that
  \[
    e^{(2a_n)^{\epsilon}} < M_{\phi_1}(2a_n) = \phi_1(2a_n i)
  \]
  for all $n$. Also since $\phi_1(z)$ has order $\rho=1$ and minimal type,
  \[
    M_{\phi_1}(4a_n) < e^{\epsilon(4a_n)}
  \]
  for all sufficiently large $a_n$. For small $\epsilon$ and large $a_n$
  the lower bound on the fraction in
  inequality~\eqref{inequality:phivanishingat0:2} is then bounded below as
  follows:
  \begin{align*}
    \frac{\cosh(2a_nr) \phi_1(2a_ni)}{2^{2m-1} \phi_1(4a_n i)}
     & > \frac{e^{2a_n r} e^{(2a_n)^{\epsilon}}}{2^{2m} e^{\epsilon(4 a_n)}} > 1.
  \end{align*}
  Therefore, similarly to the previous case in which $\rho<1$,
  $\phi(D)g_{a_n}(z)$ has roots with imaginary part $r_1>0$ where
  \begin{align*}
    r_1 & = \frac{1}{2a_n} \cosh^{-1}\left(\tfrac{\cosh(2a_nr)\phi_1(2a_ni) + \sqrt{\cosh^2(2a_nr)\phi_1^2(2a_ni)
                                               + 2^{4m-1}\phi_1^2(4a_ni)}}{2^{2m}\phi_1(4a_ni)}\right) \\
        & > \frac{1}{2a_n} \cosh^{-1}\left(\frac{e^{2a_n r} e^{(2a_n)^{\epsilon}}}{2^{2m} e^{\epsilon(4 a_n)}}\right) \\
        & > \frac{1}{2a_n} \log\left(\frac{e^{2a_n r} e^{(2a_n)^{\epsilon}}}{2^{2m} e^{\epsilon(4 a_n)}}\right) \\
        & = r+\frac{1}{2a_n} \left((2a_n)^{\epsilon}-\log(2^{2m}) - 4 \epsilon a_n\right) \\
        & > r-4\epsilon-\frac{\log(2^{2m})}{2a_n}.
  \end{align*}
  If $\epsilon$ is sufficiently small and $a_n$ is sufficiently large, the
  lower bound on $r_1$ can be made to be arbitrarily close to $r$. This
  shows that $\phi(D)$ is not a CZSDO when $\phi(z)$ has order $\rho=1$
  and has minimal type, proving Lemma~\ref{lemma:phi:vanishingat0}(b).
\end{proof}

Combining Lemmas~\ref{lemma:evenphinotCZSDO}
and~\ref{lemma:phi:vanishingat0} proves
Theorem~\ref{theorem:MainTheorem:2} parts~(a) and~(b) in the case
$\phi(0) = 0$.

We have now established Theorem~\ref{theorem:MainTheorem:2} parts~(a)
and~(b). Theorem~\ref{theorem:MainTheorem:2}(c) is
Lemma~\ref{lemma:shiftingoperator}. That is, if $\alpha \in \R$, then
$e^{\alpha D}f(z)=f(z+\alpha)$. The operator $e^{\alpha D}$ merely
translates the zeros of $f(z)$ horizontally in the complex plane and does
not reduce the size of the strip $S(r)$ containing the roots. So,
$e^{\alpha D}$ is not a CZSDO.

Finally, to obtain the necessary condition stated at the end of the
theorem, write
\[
  \phi(z) = e^{\alpha z}\phi_1(z),\qquad
  \phi_1(z) = c z^m\prod_n(1-z/a_n)e^{z/a_n}.
\]
The composition identity proved after
Lemma~\ref{lemma:differentialoperator} and
Lemma~\ref{lemma:shiftingoperator} give
\[
  \phi(D)f(z) = \big(\phi_1(D)f\big)(z+\alpha)
\]
for every $f(z)\in\LP(r)$. Since $\alpha$ is real, translations by
$\alpha$ and by $-\alpha$ preserve each class $\LP(r_1)$ by these same
lemmas. Thus $\phi(D)$ is a CZSDO if and only if $\phi_1(D)$ is a CZSDO.
The necessary condition follows by applying parts~(a) and~(b) to
$\phi_1(z)$.

The proof of Theorem~\ref{theorem:MainTheorem:2} is now complete.

\section{Examples and conjectures} \label{section:examples}

In this section, we give several examples and make some conjectures based
on these examples.

\begin{example}  \label{example:sinandcosine}
  In Lemma~\ref{lemma:cosineandsineoperators} we showed that for
  $a,b \in \R$ where $a>0$ and for any $f(z) \in \LP(r)$ with $r>0$,
  \begin{equation}
    \cos(aD+b)f(z)  \in \LP(\sqrt{r^2-a^2})
  \end{equation}
  and
  \begin{equation}
    \sin(aD+b)f(z)  \in \LP(\sqrt{r^2-a^2}).
  \end{equation}
\end{example}

\begin{example}  \label{example:exponentialsquared}
  In Lemma~\ref{lemma:exponentialsquared} we showed that for $a>0$ and any
  $f(z) \in \LP_1(r)$ with $r>0$
  \begin{equation}
    \exp\left(-\frac{a^2 D^2}{2}\right) f(z) \in \LP(\sqrt{r^2-a^2}).
  \end{equation}
\end{example}

Example~\ref{example:sinandcosine} suggests that if $\phi(z) \in \LP_1$
has a `lower density' of zeros then $\phi(aD)$ might be a CZSDO. We state
this as a conjecture:
\begin{conjecture}[False] \label{conjecture:lowerdensity}
  Let $n(t)$ denote the number of roots of $\phi(z)\in\LP_1$ in the
  interval $(-t,t)$. Suppose
  \[
    \liminf_{t \rightarrow \infty} \frac{n(t)}{t} > 0.
  \]
  Then there exists a positive constant $c_{\phi}$ such that for any
  $f(z) \in \LP(r)$ with $r>0$ it follows that
  \[
    \phi(aD)f(z) \in \LP\left(\sqrt{r^2-c_{\phi} a^2}\right).
  \]
\end{conjecture}

From the proof of Theorem~\ref{theorem:MainTheorem:2} and in light of
Examples~\ref{example:sinandcosine} and~\ref{example:exponentialsquared}
we make another conjecture:

\begin{conjecture}[Classification of $\phi(D)$ that are CZSDOs. False] \label{conjecture:MainConjecture}
  If $\phi(z)\in \LP$ has a Weierstrass canonical product of the form
  \[
    \phi(z) = c e^{\alpha z-\beta z^2} \prod_n \left(1-\frac{z}{\alpha_n}\right) \exp\left(\frac{z}{\alpha_n}\right),
  \]
  then $\phi(D)$ is a CZSDO if and only if any one of the following
  conditions is satisfied:
  \begin{enumerate}
    \item[(i)] $\beta>0$.
    \item[(ii)] $\beta=0$ and the product
      $\prod_n(1-z/\alpha_n)e^{z/\alpha_n}$ has order $\rho=1$ and type
      $\sigma>0$ or has order $\rho=1$ and maximal type.
    \item[(iii)] $\beta=0$ and the order of the product
      $\prod_n(1-z/\alpha_n)e^{z/\alpha_n}$ satisfies $1<\rho \leq 2$.
  \end{enumerate}
  In other words, we would have a complete classification of the functions
  $\phi(z) \in \LP$ such that $\phi(D)$ is a CZSDO.
\end{conjecture}

\begin{remark}
  Conjectures~\ref{conjecture:lowerdensity}
  and~\ref{conjecture:MainConjecture} are false. Consider
  \[
    \phi(z) = \prod_{n = 1}^{\infty}
    \left(1-\frac{z}{n^{2/3}}\right)e^{z/n^{2/3}}.
  \]
  Since $\sum_{n=1}^{\infty}n^{-4/3}<\infty$, this function belongs to
  $\LP_1$. Its zeros are the positive numbers $n^{2/3}$, so the number
  $n(t)$ of zeros in $(-t,t)$ satisfies $n(t)/t\rightarrow\infty$. It also
  has order $3/2$. To see this, for $R\geq 2$, let $N=\lfloor R^{3/2}\rfloor$.
  The elementary inequalities
  \[
    \log|(1-w)e^w| \leq 2|w|,
    \qquad
    \log|(1-w)e^w| \leq |w|^2/2
  \]
  follow from $|1-w|\leq 1+|w|$ and $\log u\leq u-1$, respectively. They
  give
  \begin{align*}
    \frac{N\log 2}{2}
     & \leq \log|\phi(iR)| \leq \log M_\phi(R) \\
     & \leq 2R\sum_{n \leq N}n^{-2/3}
    +\frac{R^2}{2}\sum_{n > N}n^{-4/3}
    = O(R^{3/2}).
  \end{align*}
  The first inequality follows by taking the moduli of the factors at
  $iR$. The last estimate follows by comparison with the integrals of
  $x^{-2/3}$ and $x^{-4/3}$. Thus the order is $3/2$ and the type is
  positive and finite.

  Nevertheless, $\phi(D)$ is not a CZSDO. Fix $r>0$, and for $t>0$ let
  \[
    f_t(z) = e^{-tz}(z^2+r^2)\in\LP_1(r),
    \qquad
    b(t) = \sum_{n = 1}^{\infty}\frac{1}{(t+n^{2/3})^2}.
  \]
  The product rule and the Taylor series of $\phi$ at $-t$ give
  \begin{align*}
    \phi(D)f_t(z)
     & = e^{-tz}\big(\phi(-t)(z^2+r^2)+2\phi'(-t)z+\phi''(-t)\big) \\
     & = e^{-tz}\phi(-t)
    \left(\left(z+\frac{\phi'(-t)}{\phi(-t)}\right)^2+r^2-b(t)\right),
  \end{align*}
  where the second equality follows by differentiating the logarithmic
  derivative of the canonical product. Here $\phi(-t)\neq 0$, and
  $b(t)\rightarrow 0$ by dominated convergence, since its summands tend to
  zero and are bounded by the summable sequence $n^{-4/3}$. For all
  sufficiently large $t$, the two zeros of $\phi(D)f_t$ therefore have
  imaginary parts $\pm\sqrt{r^2-b(t)}$. These approach $\pm r$, so no
  $r_1<r$ works for all $f_t$. Taking $a=1$ disproves
  Conjecture~\ref{conjecture:lowerdensity}. The same function disproves
  the sufficiency of condition~(iii) in
  Conjecture~\ref{conjecture:MainConjecture}.

  Condition~(i) remains sufficient by Theorem~\ref{theorem:MainTheorem:1}.
\end{remark}

\begin{example}[Simple zeros]
  In~\cite{Cardon2005a}, Cardon and de Gaston showed that if
  $\phi_1(z) \in
\LP_1$ has infinitely many zeros and if $f(z) \in \LP$
  then $\phi_1(D) f(z)$ has only simple real zeros. This guarantees that
  if $\phi(z)=e^{-\alpha^2
z^2/2}\phi_1(z)$ and $\alpha \geq r$, then for any
  $f(z) \in
\LP^{\alpha^2/2}(r)$, the function $\phi(D)f(z)$ has only simple
  real zeros.
\end{example}

\begin{example}[Lower bound on $r_1$] \label{example:lowerbound}
  Let $\phi(z) \in \LP$ where $\phi(z)$ is not of the form
  $ce^{\alpha z}$. By computing the derivative of the logarithmic
  derivative and using the product representation in
  Eq.~\eqref{eqn:LaguerrePolyaWeierstrassProduct} and if $z$ is real and
  not a root of $\phi$ we obtain
  \[
    \left(\frac{\phi'(z)}{\phi(z)}\right)' = \frac{\phi''(z)\phi(z)-(\phi'(z))^2}{(\phi(z))^2}
    = - \frac{m}{z^2} - 2\beta - \sum_k\frac{1}{(z-\alpha_k)^2} < 0.
  \]
  The term $m/z^2$ is omitted when $m=0$. Then for real $z$ not a root
  \begin{equation} \label{inequality:Laguerre}
    [\phi'(z)]
    ^2-\phi(z)\phi''(z) > 0.
  \end{equation}
  This is called \textit{Laguerre's inequality}. Suppose $\phi(0)=1$. Then
  the constant $b_{\phi}$ defined by
  \[
    b_{\phi} = [\phi'(0)]^2-\phi(0)\phi''(0) = [\phi'(0)]^2-\phi''(0)
  \]
  is positive.

  Letting $\phi(aD)$ act on $z^2+r^2$ where $a>0$ and $r>0$ gives
  \begin{align*}
    \phi(aD)(z^2+r^2) & = \sum_{k = 0}^{\infty} \tfrac{a^k \phi^{(k)}(0)}{k!}  \tfrac{d^k}{dz^k}(z^2+r^2) \\
                      & = \phi(0)(z^2+r^2) + a \phi'(0)(2z)+\tfrac{a^2}{2}\phi''(0) (2) \\
                      & = z^2 + 2a\phi'(0) z +r^2 + a^2\phi''(0) \\
                      & = (z+a\phi'(0))^2 + r^2 -  \big([\phi'(0)]^2-\phi''(0)\big)a^2 \\
                      & = (z+a\phi'(0))^2 + r^2-b_{\phi} a^2.
  \end{align*}
  The roots of $\phi(aD)(z^2+r^2)$ belong to the strip
  $S(\sqrt{r^2-b_{\phi}a^2})$ but they do not belong to any smaller strip.

  The calculation of $\phi(aD)(z^{2}+r^{2})$ implies that if
  $\phi \in \LP_{1}$, $\phi(0) = 1$, and $\phi(aD)$ is a CZSDO that maps
  $\LP(r)$ into $\LP(r_1)$ with $0 \leq r_1 < r$, we have a lower bound
  for $r_1$. That is, if $0 \leq a \leq r/b_{\phi}^{1/2}$, then
  \[
    r_1 \geq \sqrt{r^{2}-b_{\phi}a^{2}}.
  \]
\end{example}

In all explicit examples of CZSDOs in the paper involving the differential
operator $\phi(D)$, we see that the narrowing of the strip $S(r)$ to
$S(r')$ involved constants $b_{\phi}$ and $c_{\phi}$ where
\[
  \sqrt{\max\{r^{2}-b_{\phi}a^{2},0\}} \leq r' \leq
  \sqrt{\max\{r^{2}-c_{\phi}a^{2},0\}}.
\]
We note that Theorem~1.8 in~\cite{BrandenChasse2014} provides additional
examples of CZSDOs involving these square root expressions.

\begin{conjecture} \label{conjecture:twosidedbounds}
  Assume $\phi(D)$ is any CZSDO where $\phi(z) \in \LP$. For $r > 0$ and
  $a > 0$, let $r'_{a}$ denote the smallest number $s \geq 0$ such that
  $\phi(aD)f(z) \in \LP(s)$ for every $f(z) \in \LP_{1}(r)$. Then there
  exist positive constants $b_{\phi} \geq c_{\phi}$, depending only on
  $\phi$, such that for each $r > 0$ and $a > 0$
  \[
    \sqrt{\max\{r^{2} - b_{\phi}a^{2}, 0\}} \leq r'_{a} \leq
    \sqrt{\max\{r^{2} - c_{\phi}a^{2}, 0\}}.
  \]
\end{conjecture}

\begin{example}[Multiplier sequences]
  In this paper we defined CZSDOs fairly narrowly because we restricted
  ourselves to the case of CZSDOs of the form $\phi(D)$ for
  $\phi \in \LP$. It is easy to find examples of CZSDOs that are not of
  the form $\phi(D)$ where $\phi \in \LP$ by considering multiplier
  sequences known to preserve the reality of zeros. We recall that a
  sequence of numbers
  \[
    \gamma_0, \gamma_1, \gamma_2,\ldots
  \]
  is called a \textit{multiplier sequence of the first kind} if for any
  polynomial
  \[
    f(z) = a_0 + a_1 z + \cdots + a_n z^n
  \]
  having only real zeros, the new polynomial
  \begin{equation} \label{equation:multipliersequence}
    \Gamma[f(z)] = \gamma_0 a_0 + \gamma_1 a_1 z + \cdots + \gamma_n a_n z^n
  \end{equation}
  also has only real zeros. Important classical results about multiplier
  sequences were proved by P\'olya and Schur. For a nice discussion about
  the classical results see Rahman and
  Schmeisser~\cite{RahmanSchmeisser2002} Section 5.7 and
  Levin~\cite{Levin1980} Chapter VIII Section 3. Multiplier sequences are
  still a topic of investigation as can be seen in papers of Craven and
  Csordas~\cite{Craven1996,Craven2004,Craven2006}, just to name a few.
  Consider the multiplier sequence given by
  $\Gamma=\{\gamma_k\}=\{\alpha^k\}$ where $1<\alpha$. If for $r>0$,
  \[
    f(z) = \sum_{n = 0}^{\infty} a_n z^n \in \LP(r)
  \]
  then
  \[
    \Gamma[f(z)] = \sum_{n = 0}^{\infty} a_n (\alpha z)^n \in \LP(r/\alpha).
  \]
  So, this is an example of a CZSDO. However, if $0<\alpha<1$, then this
  multiplier sequence increases the size of the strip containing the zeros
  and is not a CZSDO.
\end{example}

\section{Questions for further study} \label{section:futherstudy}

In addition to the conjectures stated in the previous section, we end this
paper with several open problems concerning complex zero strip decreasing
operators.

\begin{openproblem}
  Completely classify the CZSDOs of the form $\phi(D)$, which is the type
  studied in this paper.
\end{openproblem}

\begin{openproblem}
  \emergencystretch=1em Completely classify the CZSDOs resulting from
  $\Gamma$-sequences as in Eq.~\eqref{equation:multipliersequence}.
\end{openproblem}

\begin{openproblem}
  Completely classify CZSDOs when the class being acted on is the class of
  all real polynomials whose roots belong to the region $\Omega=S(r)$ for
  $r>0$. Such a classification might be in the style of the classification
  theorems in~\cite{BorceaBranden2009c} and~\cite{BrandenChasse2014},
  which were mentioned in the introduction.

\end{openproblem}

\section*{Acknowledgments}
I wish to thank my PhD advisor Daniel Bump who introduced me to many
beautiful ideas of analytic number theory. As I was writing my PhD thesis,
he showed me P\'olya's 1926 paper~\cite{Polya1926} on the Riemann zeta
function. An idea from that paper sparked my interest and is the genesis
for this paper as well as several previous papers. I acknowledge George
Csordas and Tom Craven for their very interesting papers which have
inspired me to work on problems in this area. I also thank the reviewer
for helpful comments.


\end{document}